\documentclass{article}
\usepackage[margin = 1.5in]{geometry}

\usepackage{graphicx} 
\usepackage{amsthm, thmtools}
\usepackage{commath, mathtools, amsfonts, amsmath, amssymb}
\usepackage{tikz-cd}
\usepackage{thmtools}
\usepackage[algoruled]{algorithm2e}

\usepackage{authblk}
\usepackage{fontawesome5}

\usepackage[
    backend=biber,
    style=numeric,
    maxnames=19
]{biblatex}

\usepackage{hyperref}
\hypersetup{
    colorlinks=true,
    linkcolor=blue,
    filecolor=blue,      
    urlcolor=blue,
    citecolor=blue, 
    pdftitle={Base change lifts},
    pdfpagemode=FullScreen,
    }

\theoremstyle{definition}
\newtheorem{theorem}{Theorem}[section]

\newtheorem{remark}{Remark}[section]
\newtheorem{lemma}{Lemma}[section]

\declaretheoremstyle[
    spaceabove=1em, spacebelow=1em,
    headfont=\normalfont\bfseries,
    notefont=\mdseries, notebraces={(}{)},
    bodyfont=\normalfont,
    postheadspace=0.5em,
    qed=$\blacklozenge$
]{examplestyle}

\declaretheorem[numberwithin = section, style = examplestyle]{example}

\declaretheorem[numberwithin = section, style = examplestyle]{definition}

\newcommand*{\C}{\mathbb{C}}
\newcommand*{\Q}{\mathbb{Q}}
\newcommand*{\R}{\mathbb{R}}

\newcommand*{\Z}{\mathbb{Z}}
\newcommand*{\F}{\mathbb{F}}
\newcommand*{\A}{\mathbb{A}}

\newcommand{\GL}{\operatorname{GL}}

\newcommand{\Fr}{\operatorname{Fr}}

\renewcommand*{\O}{\mathcal{O}}

\newcommand*{\Gal}{\operatorname{Gal}}
\newcommand{\Tr}{{\operatorname{Tr}}}
\newcommand{\Norm}{\operatorname{N}}

\newcommand{\Spec}{\operatorname{Spec}}
\newcommand{\cl}{\operatorname{cl}}

\newcommand{\fk}[1]{\mathfrak{#1}}
\newcommand{\rhoreg}{\rho_{\operatorname{reg}}}

\title{On the computation of base-change lifts and lifts of Hida families}
\author[1]{Iván Blanco-Chacón}
\author[2]{Luis Dieulefait}
\author[3]{Antti Haavikko}

\affil[1,3]{Universidad de Alcal\'a, Alcal\'a de Henares, Spain}
\affil[2]{Universitat de Barcelona, Barcelona, Spain}

\begin{document}
\maketitle

\begin{abstract}
We derive an explicit formula for the Hecke eigenvalues of a Hilbert modular form which is a base-change lift of a classical newform to a totally real number field. We show that for a totally real Galois number field $F$ the $L$-function of a base-change lifted form can be factorized as a product of twists of the $L$-function of the underlying classical form over irreducible representations of $\Gal(F / \Q)$. Moreover, we use the formula for the Hecke eigenvalues of a base-change lift to prove the existence of a base-change lift of a Hida family. In particular, we show that a Hida family of classical Hecke eigenforms can be lifted to a formal power series that specializes to the base-change lifts of the Hida family of classical cusp forms.
\end{abstract}


\vspace{2mm}
\hrule
\vspace{2mm}
\textit{\faEnvelope[regular] Corresponding author: $^3$antti.haavikko@edu.uah.es}

\section{Introduction}
Given a number field $F$, automorphic forms for $\GL_n / F$ generalize classical modular forms ($n=2$ and $F=\Q$), Hilbert modular forms ($n=2$ and $F$ a totally real number field), and other classical objects like Bianchi modular forms and Shimura automorphic forms. As in the case of classical and Hilbert modular forms, automorphic forms yield automorphic representations $\pi: \GL_n(\A_F) \to \GL(V^{\pi})$, where $\A_F$ is the adele ring of $F$ and $V^{\pi}$ is a certain Hilbert space attached to $\pi$. For a precise definition of automorphic representations and a detailed overview of their properties, we refer the reader to \cite[pp.~165-170]{cornelletal} and \cite{getz2019introduction}.

A Galois representation $\rho: G_F \to \GL_n(R)$, where $R$ is a coefficient ring, is hence called \emph{automorphic} if there exists an automorphic representation $\pi:\GL_n(\A_F)\to \GL(V^{\pi})$ such that for every place $\nu$ of $F$ where $\rho$ is unramified, one has that $\rho(\Fr_{\nu}) = t_{\pi_{\nu}}$, where $\Fr_{\nu}$ is a Frobenius element at $\nu$, $\pi_{\nu}$ is the local automorphic representation $\pi|_{F_{\nu}}$ and $t_{\pi_{\nu}}$ its Langlands class.

One of the many facets of the Langlands functoriality conjecture (see \cite[p.~185]{cornelletal}) is the base-change problem (see \cite[p.~192]{cornelletal}). In its essence, the base-change problem is the following: given an automorphic Galois representation $\rho: G_F \to \GL_n(R)$ and an extension $L / F$, is it true that the restriction $\rho|_{G_L}$ is also automorphic for $\GL_n / L$?

For general extensions of number fields $L / F$, the problem remains open but some major contributions have been established for particular types of extensions. For instance, the base-change problem was solved by Langlands \cite{langlands} when $\Gal(L/F)$ is abelian, and by Arthur and Clozel \cite{arthurclozel} when $\Gal(L/F)$ is solvable. Later on, Dieulefait \cite{dieulefait1} solved the base-change problem for $\mathrm{GL}_2$ with $F=\Q$ and $L$ a totally real Galois number field, removing the condition of $L/\Q$ being solvable. In \cite{dieulefait2}, a simplified proof is given, together with a proof of the automorphy of $\mathrm{Symm}^5(\pi_f)$ for a modular cusp form $f$, where $\pi_f$ stands for its attached automorphic representation. They also manage to lift the assumption on the extension $L / \Q$ being Galois. See also \cite{blancodieulefait} for a corrigenda of \cite{dieulefait2}. 

These aforementioned works are highly non-constructive in nature. They show the existence of the automorphic lifted representation, which is determined uniquely by local properties of the automorphic representation of the base field, but they do not describe the geometric object they are attached to. For fields that are not totally real nor CM, it is not even clear which geometric objects these automorphic representations are attached to, if any. By describing the geometric object we mean giving explicitly the Hecke eigenvalues uniquely attached to the lifted automorphic forms and showing how these can be explicitly obtained from those of the geometric object attached to the automorphic representation of the base field.

The base-change problem was first addressed from a more explicit and computational perspective by Doi and Naganuma in \cite{doi1969functional}, earlier than \cite{langlands}, \cite{arthurclozel} and \cite{dieulefait1}. Doi and Naganuma started with a classical Hecke eigenform $f$ and a quadratic real number field $F = \Q(\sqrt{D})$ under some conditions on the nebentypus character and the field $F$. They fully described the Hilbert modular form $h$ associated to the lift to $F$ of the Deligne representation associated to $f$. In particular, they gave the Hecke eigenvalues of $h$ in terms of those of $f$ via a numerical recipe which allows one to compute the eigenvalues of $h$ knowing only those of $f$. Later in \cite{saito}, Saito generalized Doi and Naganuma's result to totally real number fields $F$ when $\Gal(F/\Q)$ is cyclic of prime order and $F$ has class number 1.

The first goal of the this work is to generalize the explicit approach by Doi, Naganuma, and Saito to any totally real number field and derive a formula for the Hecke eigenvalues of the base-change lift $h$ in terms of the Hecke eigenvalues of the classical form $f$. Our second goal is to give a base-change result for Hida families. We will use the explicit base-change formulas derived for classical newforms to construct the base change of a Hida family of ordinary newforms. This base change will be a Hida family of ordinary Hilbert modular forms. 

In Section \ref{sec: base-change lifts of classical modular forms}, we recall the concept of base change from classical newforms to Hilbert modular forms for totally real neumber fields as well as the precise statement of the existence results in \cite{dieulefait1} and \cite{dieulefait2}. Section \ref{sec: Hecke eigenvalues of base-change lifts} generalizes the results of Doi, Naganuma, and Saito and derives a formula to compute the Hecke eigenvalues of a base-change lift to any totally real Galois number field. Note that the existence of the lift is due to \cite{dieulefait1} and \cite{dieulefait2}.

Section \ref{sec: Factorization of L-series of base-change forms} proves that the $L$-function of a base-change lift $h$ of $f$ to a totally real Galois number field $F$ has a factorization
\begin{align*}
    L(h,s) = \prod_{\rho} L(f \otimes \rho, s)^{\dim(\rho)},
\end{align*}
where $\rho$ are the irreducible representations of $\Gal(F / \Q)$. Specifically, we show that this factorization holds for all but finitely many local Euler factors at primes $\fk{p} \subset \O_F$ over primes $p$ not dividing the level of $f$.

In Section \ref{sec: base-change lifts of Hida families}, we recall the concepts of Hida families of ordinary classical cusp forms and Hilbert cusp forms and the main theorems which grant a Hida family passing by the input datum at its weight. Finally, we define the base-change lift of a Hida family of classical Hecke eigenforms and prove the existence of these lifts by an explicit construction using the formulas obtained in Section \ref{sec: Hecke eigenvalues of base-change lifts}. The formulas must be applied to the Iwasawa coefficients of the formal power series to show that the lifted formal power series specialize to Hilbert modular forms which are precisely the base-change lifts of the specializations of the Hida family of classical cusp forms. 

As an application of the base-change lift of Hida families, we generalize one of the two main results in \cite{blancodieulefait}. We prove the following result for ordinary classical cusp forms.
\begin{theorem}
Let $f\in S_k(N,\chi)$ be a $p$-ordinary cuspidal Hecke eigenform, $\mathcal{O}_{K_f}$ the ring of integers of its field of definition and $\fk{p}$ a prime of $\mathcal{O}_{K_f}$ above $p \nmid N$. Then, for each totally real number field $F$, there exists a sequence $\{k_r\}_{r\geq 1} \subseteq \mathbb{N}$ such that $\overline{\rho}_{f,\fk{p}}|_{G_F}$ admits a potentially diagonalizable automorphic lift of Hodge--Tate weights $\{0,k_r-1\}$.    
\end{theorem}

Notice that the advantage of using Hida families is that it allows us to remove the condition $p>\max\{k,6\}$, the condition of not being a CM form and the condition on the residual representation having a large image that all appear in \cite{blancodieulefait}. 

We conclude Section \ref{sec: base-change lifts of Hida families} with the corresponding result in the non-ordinary setting:
\begin{theorem}
Let $f \in S_k(N,\chi)$ be a non-ordinary cuspidal Hecke eigenform, $\mathcal{O}_{K_f}$ the ring of integers of its field of definition and $\fk{p}$ a prime of $\mathcal{O}_{K_f}$ above $p \nmid N$. Assume that $\overline{\rho}_{f,\fk{p}}|_{G_F}$ has large image, $f$ is not a CM form, and $p>\max\{k,6\}$. Then, for each totally real number field $F$, there exists a sequence $\{k_r\}_{r\geq 1} \subseteq \mathbb{N}$ such that $\overline{\rho}_{f,\fk{p}}|_{G_F}$ admits a potentially diagonalizable automorphic lift of Hodge--Tate weights $\{0,k_r-1\}$.      
\end{theorem}

Lastly, in Section \ref{sec: implementation in Magma} we present the pseudocode for an algorithm that computes the Hecke eigenvalues of the base-change lift as described in Section \ref{sec: Hecke eigenvalues of base-change lifts}. This allows us to explicitly compute the coefficients of the base-change lift and to pinpoint exactly the lift of a classical modular form to any totally real number field. The Magma implementations of the algorithms of Section \ref{sec: implementation in Magma} can be found in Appendix \ref{appendix: Magma implementations}.

\section{Base-change lifts of classical modular forms}
\label{sec: base-change lifts of classical modular forms}
For a number field $F$, let us denote by $G_F=\Gal(\overline{F}/F)$ the absolute Galois group of $F$. For any place $\nu$ of $F$, we fix algebraic closures $\overline{F}$ and $\overline{F}_\nu$ to get an embedding $G_{F_{\nu}}\hookrightarrow G_F$, where $F_{\nu}$ is the completion of $F$ at $\nu$. If $L/F$ is a Galois extension of number fields, it is clear that $G_L\subseteq G_F$. Let $\rho: G_F \to \GL_2(R)$ be a continuous Galois representation over a ring $R$, which can be either the ring of integers of a number field, a local field or a finite field. For a place $\nu$ of $F$ such that $\rho$ is unramified we denote $\rho_{\nu} := \rho|_{D_{\nu}}$, where $D_{\nu}$ is the decomposition group at $\nu$. Lastly, for any field $F$, we will write $\O_F$ for the ring of integers of $F$, $\A_F$ for the ring of adeles of $F$ and $\nu$ for a place of $F$.

To discuss the concept of base change, we need a few classical results concerning representations associated to modular forms. The following result by Deligne and Serre \cite[Theorem 6.1]{serre1974formes} is well known.
\begin{theorem}
    \label{thm:trace and determinant of Deligne representation}
    Let $f \in S_k(N, \chi)$ be a classical newform of level $N$ and nebentypus $\chi$ and denote the coefficient field of $f$ by $K_f$. Finally, let $q$ be a rational prime and let $\fk{q} \subset \O_{K_f}$ be a prime ideal above $q$. Then, there exists a Galois representation
    \begin{align*}
        \rho_{f,\fk{q}}: G_\Q \to \GL_2(\O_{K_f, \fk{q}})
    \end{align*}
    unramified at all primes $p \nmid Nq$. Moreover, the representation satisfies the properties
    \begin{align*}
        \Tr(\rho_{f,\fk{q}}(\Fr_p)) &= a_p \quad \text{ and } \quad \det(\rho_{f,\fk{q}}(\Fr_p)) = \chi(p) p^{k-1},
    \end{align*}
    where $a_p$ is the $p$-th Fourier coefficient of $f$, or equivalently, the Hecke eigenvalue of $f$ at $p$.
\end{theorem}

For the following definition, we keep the notation of the previous theorem.
\begin{definition}
Let $F$ be a number field. If the restriction $\rho_{f,\fk{q}}|_{G_F}$ is automorphic, that is, it is attached to an automorphic representation 
$$\pi_{f,F}= \bigotimes_{\nu} \pi_{\nu}$$
of $\GL_2(\mathbb{A}_F)$, we say that $\pi$ is an \emph{automorphic base-change lift} of $\rho_{f,\fk{q}}$ (or $f$) to $F$.
\end{definition}
Notice that in this case, we have that for each place $\nu$ of $F$ not dividing $N$
$$
\rho_{f,\fk{q}}|_{D_{\nu}}(\Fr_{\nu})=t_{\pi_{\nu}},
$$
where $D_{\nu}$ is the decomposition group at $\nu$ and $t_{\pi_{\nu}}$ is the local Langlands class at $\nu$ (see \cite[Chapter VI]{cornelletal} for details).

The relevant case to us is when $F$ is a totally real number field. In this setting Hilbert modular cusp forms, like classical modular cusp forms, provide a source of automorphic representations. Indeed, following the notation in the write-up by Dimitrov \cite{dimitrov2013arithmetic}, there is a canonical bijection between Hilbert newforms $h$ in $S_{k,w_0}(K_1(\fk{N}),\psi)$ and cuspidal automorphic representations $\pi$ of $\GL_2(\A_F)$ of conductor $\fk{N} \subset \O_F$, central character $\psi$, and where the Archimedean representation $\pi_\infty$ belongs to the holomorphic discrete series of arithmetic weights $(k,w_0)$ (see \cite[p.~97]{bump} and \cite{dimitrov2013arithmetic}). It is also possible to attach Galois representations to Hilbert cusp forms by generalizing Theorem \ref{thm:trace and determinant of Deligne representation} (see Carayol \cite{carayol1986representations}, Taylor \cite{taylor1989galois}, Blasius and Rogawski \cite{blasius1989galois}, and Wiles \cite{wiles1986p}). For a survey on the known results, see Jarvis \cite{jarvis1997galois}.

The classical Serre's modularity conjecture over $\Q$ states that for any odd and absolutely irreducible $2$-dimensional representation $\overline{\rho}$ of $G_\Q$ over a finite field, one can find a classical eigenform $f$ such that the residual Deligne representation $\overline{\rho}_{f, \fk{q}}$ (see Theorem \ref{thm:trace and determinant of Deligne representation}) and $\overline{\rho}$ are equivalent. This was first proved for level 1 and weight 2 by Dieulefait \cite{dieulefait2007level}, and later in full generality by Khare and Wintenberger \cite{khare2009serreI, khare2009serreII}.

The generalized Serre's modularity conjecture for totally real number fields $F$ predicts that odd and absolutely irreducible residual representations of $G_F$ are attached to Hilbert modular forms and hence to automorphic representations. For the cases where the residual representation takes values in $\F_2$ or $\F_3$, the generalized conjecture follows from the work of Langlands--Tunnell \cite{langlands, tunnell1981artin}. The generalized Serre's conjecture has also been established for representations taking values in $\mathbb{F}_4$ (\cite{taylorsbarron}), $\mathbb{F}_5$ (\cite{taylor2003}), $\mathbb{F}_7$ (\cite{manoharmayum}) and $\mathbb{F}_9$ (\cite{ellenberg}). The generalized conjecture is still a very active and rich area of research.

Theorem \ref{thm:trace and determinant of Deligne representation} together with the results in \cite{carayol1986representations}, \cite{blasius1989galois} and \cite{wiles1988} suggest a strong connection between the automorphic representations attached to classical cusp forms and those attached to Hilbert modular forms. This connection is one of the objects of the Langlands functoriality conjectures, namely, the base-change problem, which for totally real number fields $F$ can be stated as follows:
\begin{definition}
    Let $f$ be a classical newform and denote by $\rho_{f,\fk{q}}$ the Galois representation attached to $f$. We call the representation $\rho_{h,\fk{q}}$ attached to a Hilbert modular form $h$ \emph{a base-change lift of $\rho_{f,\fk{q}}$ to $F$} if the restriction 
    $$
    \rho_{f,\fk{q}}|_{G_F} : G_F \to \GL_2(\O_{{K_f}, \fk{q}})
    $$
    is equivalent to the representation
    $$
    \rho_{h,\fk{q}}: G_F \to \GL_2(\O_{{K_h}, \fk{q}}),
    $$
    where $K_h$ is the Hecke eigenvalue field of $h$. Moreover, we say that the Hilbert modular form $h$ is the \emph{base-change lift of $f$ to $F$}.
\end{definition}

We have the following strong result regarding the existence of the base-change lifts of classical newforms to totally real number fields $F$.
\begin{theorem}
\label{thm: existance of base-change lifts to F}
    Let $F$ be a totally real number field and let $f$ be a newform of any positive level and weight $k \geq 2$. Then there exists a Hilbert modular form $h$ that is a base-change lift of $f$ to $F$.
\end{theorem}
\begin{proof}
    The result was first proved by Dieulefait \cite{dieulefait1} for newforms of odd level and weight $k \geq 2$ together with some conditions on the splitting behaviour of a small set of primes in $F$. However, these conditions were lifted and the theorem was proved in its full generality as stated above in a subsequent work by Dieulefait \cite[Section 5]{dieulefait2}.
\end{proof}

Apart from the case where $[F : \mathbb{Q}]$ is prime, which was solved in \cite{doi1969functional} and \cite{saito1976lifting}, the general proof of this result is non-constructive. In the next section, we develop a method to compute explicitly the base-change lift of a classical newform to a totally real number field.

\section{Hecke eigenvalues of base-change lifts}
\label{sec: Hecke eigenvalues of base-change lifts}
For this section, we fix a newform $f \in S_k(\Gamma_1(N), \chi)^{\operatorname{new}}$ of weight $k \geq 2$, level $N$, and nebentypus $\chi$ with the Fourier expansion
\begin{align*}
    f(z) &= \sum_{n = 1}^\infty a(n) e^{2 \pi i n z}.
\end{align*}

In the proof of the main theorem (Theorem \ref{thm:hecke eigenvalues of a base-change lift}) of this section, we need the following lemma that unravels the recursion in the traces of powers of matrices
\begin{align*}
       \Tr(A^{n+1}) &= \Tr(A)\Tr(A^n) - \det(A)\Tr(A^{n-1})
\end{align*}
together with the recursive formula
\begin{align*}
    a(p^{n}) &= a(p)a(p^{n-1}) - \chi(p)p^{k-1} a(p^{n-2})
\end{align*}
for the Fourier coefficients of $f$ at consequent powers of $p$.

\begin{lemma}
\label{lemma:trace_of_frobenius_power}
Let $\rho_{f,\fk{q}}$ be the Galois representation attached to $f$, where $\fk{q}$ lies above some rational prime $q \nmid N$. Fix a prime $p$ not dividing $qN$. Then, for any integer $n \geq 2$
\begin{align*}
    \Tr(\rho_{f,\fk{q}}(\Fr_p)^n) = a(p^n) - \chi(p) p^{k-1} a(p^{n - 2}).
\end{align*}
\end{lemma}
\begin{proof}
The proof is by induction on $n$. Denote by $A$ the image $\rho_{f,\fk{q}}(\Fr_p)$, where $\Fr_p$ is any Frobenius element at $p$. We know from Theorem \ref{thm:trace and determinant of Deligne representation} that 
\begin{align*}
    \Tr(A) &= \Tr(\rho_{f,\fk{q}}(\Fr_p)) = a(p) \quad \text{ and } \quad \det(A) = \det(\rho_{f,r}(\Fr_p)) = \chi(p)p^{k-1}.
\end{align*}
The base case $n = 2$ is clear, since
\begin{align*}
    a(p^2) - \chi(p) p^{k-1} &= a(p)^2 - 2\chi(p) p^{k-1}
    = \Tr(A)^2 - 2 \det (A) = \Tr(A^2).
\end{align*}
In fact, this is exactly the formula presented by Doi and Naganuma for quadratic fields \cite{doi1969functional, naganuma1973coincidence}.

For the induction step, we assume that the statement holds for all integers greater or equal to 2 up to $n$. Then for $n+1$, we get
\begin{align*}
    &a(p^{n+1}) - \chi(p) p^{k-1} a(p^{n-1}) \\
    &= a(p)a(p^n) - \chi(p) p^{k-1} a(p^{n-1}) - \chi(p) p^{k-1} a(p^{n-1}) \\
    &= a(p)a(p^n) - \chi(p) p^{k-1}\left(a(p) a(p^{n-2}) - \chi(p)p^{k-1} a(p^{n-3})\right) - \chi(p) p^{k-1} a(p^{n-1}) \\
    &= a(p)\left(a(p^n) - \chi(p) p^{k-1} a(p^{n-2})\right) - \chi(p)p^{k-1} \left(a(p^{n-1}) - \chi(p)p^{k-1}a(p^{n-3})\right) \\
    &= \Tr(A)\Tr(A^n) - \det(A)\Tr(A^{n-1}) \\
    &= \Tr(A^{n+1}).
\end{align*}
\end{proof}

\begin{theorem}
\label{thm:hecke eigenvalues of a base-change lift}
    Let $F$ be a totally real number field. Denote by $h$ the base-change lift of $f$ to $F$ from Theorem \ref{thm: existance of base-change lifts to F}. Then for a rational prime $p$ not dividing the level $N$, the Hecke eigenvalue of $h$ at a prime $\fk{p} \subset \O_F $ above $p$ is 
    \begin{align*}
        C(\fk{p}) = \begin{cases}
            a(p) & \text{if } r = 1 \\
            a(p^r) - \chi(p)p^{k-1} a(p^{r-2}) & \text{otherwise},
        \end{cases}
    \end{align*}
    where $r = f(\fk{p} \mid p)$ is the residual degree of $\fk{p}$ over $p$.
\end{theorem}
\begin{proof}
The proof follows from the study of the Frobenius elements $\Fr_\fk{p} \in G_{F}$ at primes $\fk{p} \subset \O_F$ when viewed as an element of the absolute Galois group $G_\Q$. 

Let us fix a prime $p$ not dividing $N$ and take any prime ideal $\fk{p} \subset \O_F$ above $p$. Then we know that the associated Galois representation $\rho_{f,\fk{q}}$ for $p \notin \fk{q}$ is unramified at $p$ and the image of the inertia subgroup $I_p$ under $\rho_{f,\fk{q}}$ is trivial. Moreover, we have an exact sequence
\begin{align}
\label{eq: exact_seq1}
    0 \to I_p \to D_p \xrightarrow{\pi_p} G_{\F_p} \to 0.
\end{align}
The map $\pi_p$ is given componentwise for each finite Galois extension $K / \Q_p$ as
\begin{align*}
    \pi_p : D_p &\to G_{\F_p} \\
    (\sigma)_K &\mapsto \pi_p(\sigma): \ \left(x +\fk{P}_K \mapsto \sigma(x) + \fk{P}_K\right)_K,
\end{align*}
where $\fk{P}_K = \fk{P} \cap \O_K$ and $\fk{P} \in \overline{\Z}_p$ is the maximal ideal above $p$.

Similarly, we get the exact sequence
\begin{align}
 \label{eq: exact_seq2}
    0 \to I_\fk{p} \to D_{\fk{p}} \xrightarrow{\pi_\fk{p}} G_{\F(\fk{p})} \to 0,
\end{align}
where $\F(\fk{p})$ denotes the finite field $\O_F / \fk{p}$. Here the map
$\pi_\fk{p}$ is given analogously by
\begin{align*}
    \pi_\fk{p} : D_{\fk{p}} &\to G_{\F(\fk{p})} \\
    (\sigma)_K &\mapsto \pi_\fk{p}(\sigma): \ \left(x +\fk{P}_K \mapsto \sigma(x) + \fk{P}_K\right)_K,
\end{align*}
where $K$ runs over finite $p$-adic Galois extension $K / F_{\fk{p}}$ and $x$ in $\O_K$.

Moreover, since the inertia subgroup at $\fk{p}$ is defined as the inverse limit
\begin{align*}
    I_\fk{p} = \varprojlim_{K / F_{\fk{p}}}\, I(\fk{P}_K \mid \fk{p})
\end{align*}
running over finite $p$-adic Galois extensions of $F_\fk{p}$, we have the inclusion $I_\fk{p} \xhookrightarrow{} I_p$. An identical argument gives us an inclusion of the decomposition group $D_{\fk{p}}$ into $D_p$. Moreover, $\F(\fk{p})$ is an extension of $\F_p$, yielding yet another inclusion $G_{\F(\fk{p})} \xhookrightarrow{} G_{\F_p}$. These inclusions together with the two exact sequences \eqref{eq: exact_seq1} and \eqref{eq: exact_seq2} give us the following commutative diagram.
\[\begin{tikzcd}
	0 & {I_\fk{p}} & {D_{\fk{p}}} & {G_{\F(\fk{p})}} & 0 \\
	0 & {I_p} & {D_p} & {G_{\F_p}} & 0
	\arrow[from=1-1, to=1-2]
	\arrow[from=1-2, to=1-3]
	\arrow["{\pi_\fk{p}}", from=1-3, to=1-4]
	\arrow[from=1-3, to=2-3]
	\arrow[from=1-4, to=1-5]
	\arrow[from=1-4, to=2-4]
	\arrow[from=1-2, to=2-2]
	\arrow[from=2-1, to=2-2]
	\arrow[from=2-2, to=2-3]
	\arrow["{\pi_p}", from=2-3, to=2-4]
	\arrow[from=2-4, to=2-5]
\end{tikzcd}\]

Consider the arithmetic Frobenius at $\fk{p}$
\begin{align*}
    \phi_\fk{p}: \overline{\F(\fk{p})} &\to \overline{\F(\fk{p})} \\
    x &\mapsto x^{r},
\end{align*}
where $r$ denotes the residual degree $[\F(\fk{p}) : \F_p]$. We shall denote by $\Fr_\fk{p}$ any preimage under $\pi_\fk{p}$ of the arithmetic Frobenius $\phi_\fk{p}$, which is the topological generator of $G_{\F(\fk{p})}$, and call $\Fr_\fk{p}$ the Frobenius element at $\fk{p}$.

We are interested in the image of $\Fr_\fk{p} \in D_\fk{p}$ in $D_p$ under the aforementioned inclusion and the relation of this image to the Frobenius element $\Fr_p \in D_p$. Clearly, the image of the generator $\phi_\fk{p}$ in $G_{\F_p}$ equals $\phi_p^{r}$. We note that $\pi_\fk{p}^{-1}(\phi_{\fk{p}}) = \Fr_\fk{p} \circ I_\fk{p}$ and that $I_{\fk{p}} \xhookrightarrow{} I_p$. Therefore, by the commutativity of the diagram above, we see that
\begin{align*}
    \pi_p (\Fr_{\fk{p}} \circ I_{\fk{p}}) &= \pi_p (\pi_{\fk{p}}^{-1}(\phi_{\fk{p}})) = \phi_p^{r}.
\end{align*}
This implies that the image of $\Fr_\fk{p}$ in $D_p$ equals $\Fr_p^{r}$ up to inertia.

Now, we use the representations attached to $f$ and to the base change lift $h$  of $f$ to $F$ to see that the Hecke eigenvalue $C(\fk{p})$ of $h$ at $\fk{p}$ must be
\begin{align*}
    C(\fk{p}) = \Tr(\rho_{h,\fk{q}}(\Fr_{\fk{p}})) = \Tr(\rho_{f,\fk{q}}(\Fr_{p}^r)) = \Tr(\rho_{f,\fk{q}}(\Fr_{p})^r).
\end{align*}
The first equality follows from the properties of the Galois representation attached to $h$ (see \cite{carayol1986representations,taylor1989galois,blasius1989galois,wiles1986p}). Finally, the 
result follows by applying Lemma \ref{lemma:trace_of_frobenius_power} with $n = r$.
\end{proof}

\section{Factorization of L-series of base-change lifts}
\label{sec: Factorization of L-series of base-change forms}
We fix a newform $f \in S_k(\Gamma_1(N), \chi)^{\operatorname{new}}$ of weight $k \geq 2$, level $N$ and nebentypus $\chi$ with the Fourier expansion
\begin{align*}
    f(z) &= \sum_{n = 1}^\infty a(n) e^{2 \pi i n z}.
\end{align*}
We let $h$ denote the base-change lift of $f$ to a totally real Galois number field $F$ of degree $n$. The simplifying assumption that $F$ is Galois allows us to write directly the Galois group of the extension $F / \Q$ without the need of taking normal closures. In this section, our goal is to prove the following theorem.
\begin{theorem}
    \label{thm: factorization of L(h,s)}
    Let $\rhoreg : \Gal(F / \Q) \to \GL_n(\C)$ be the regular representation of $\Gal(F / \Q)$. Then up to a finite number of Euler factors that occur at primes in the level the $L$-function of $h$ factorizes as
    \begin{align}
        \label{eq: nonabelian as tensor with regular rep}
        L(h, s) = L(f \otimes \rhoreg, s).
    \end{align}
\end{theorem}

A similar theorem was stated earlier in the special case of $L$-functions of base-change lifts for totally real extensions of prime degree and with trivial level by Saito in \cite{saito1976lifting}, where the proof is omitted. Later Arthur and Clozel \cite{arthurclozel} proved an analogous factorization for the Artin $L$-functions of base-changed automorphic representations to any cyclic extension $F / L$ of prime degree. Hence, Theorem \ref{thm: factorization of L(h,s)} is a generalization of the aforementioned results \cite{saito1976lifting, arthurclozel} to any totally real Galois number fields $F$. Moreover, we give an elementary and constructive proof.

\begin{remark}
    Note that in the case that $F / \Q$ is abelian, the theorem above reduces to a a factorization over the Dirichlet characters $\chi_i$ with $i = 0, \dots, n-1$ associated to $F$, that is,
    \begin{align*}
        L(h, s) = \prod_{i=0}^{n-1} L(f, s, \chi_i),
    \end{align*}
    where
    \begin{align*}
        L(f, s, \chi_i) := \sum_{n = 1}^\infty \chi_i(n) a(n) n^{-s}
    \end{align*}
    is the twisted $L$-function of $f$.
\end{remark}

\begin{proof}
    We will denote the Galois representations associated to $f$ by $\rho_{f, \fk{q}}$ for some prime $\fk{q}$ not in the level. Let us take a rational prime $p \nmid N$ such that $p$ is not ramified at $F$. We write $p \O_F = \fk{p}_1 \dots \fk{p}_g$ for the prime factors above $p$ and $r = f(\fk{p} \mid p)$ by $r$ for the residual degree. The inverse of the local Euler factor above $p$ on the left-hand side of \eqref{eq: nonabelian as tensor with regular rep} is
    \begin{align*}
        (1 - C(\fk{p})p^{-rs} + \psi(\fk{p}) p^{-r(k-1-2s)})^g,
    \end{align*}
    since there are $g$ distinct primes $\fk{p} \mid p$ of norm $p^r$.

    To simplify the notation, we use the change of variables $x = p^{-rs}$. The corresponding local factor on the right-hand side of \eqref{eq: nonabelian as tensor with regular rep} is given by
    \begin{align*}
        \det(I - x (\rhoreg \otimes \rho_{f,\fk{q}})(\Fr_p)) &= \det(I - x\,  \rhoreg(\Fr_p) \otimes \rho_{f,\fk{q}}(\Fr_p)) \\
        &= \prod_{i=1}^{2} \prod_{j=1}^{n}(1 - \lambda_i \mu_j x),
    \end{align*}
    where $\mu_j$ denote the eigenvalues of $\rhoreg(\Fr_p)$ and $\lambda_i$ those of $\rho_{f,\fk{q}}(\Fr_p)$. 
    
    Note that since $\rhoreg$ is the regular representation of $\Gal(F / \Q)$ and $\Fr_p$ has order $r$ as an element of $\Gal(F / \Q)$, we know that $\mu_j = \zeta_r^{j}$ with $j = 1, \dots, n-1$ when $\rhoreg(\Fr_p) \neq 1$. Thus, in the case $\rhoreg(\Fr_p) \neq 1$, the eigenvalue $\mu_j$ has algebraic multiplicity $n / r = g$. Therefore, the expression above simplies to
    \begin{align*}
        \prod_{i=1}^{2} \prod_{j=1}^{n}(1 - \lambda_i \mu_j) &= \prod_{i=1}^{2} \prod_{j=1}^{r}(1 - \lambda_i \mu_j x)^g \\
        &= \prod_{i=1}^{2} \prod_{j=1}^{r}(1 - \zeta_r^j \lambda_i x)^g \\
        &= \prod_{i=1}^2 (1 - (\lambda_i x)^r)^g \\
        &= \left(1 - (\lambda_1^r + \lambda_2^r)x^r + (\lambda_1\lambda_2)^r x^{2r}\right)^g \\
        &= \left(1 - (\Tr(\rho_{f,\fk{q}}(\Fr_p^r)))x^r + (\det(\rho_{f,\fk{q}}(\Fr_p)))^r x^{2r}\right)^g.
    \end{align*}
    Now taking $g$-th roots of the left and right-hand sides and using $C(\fk{p}) = \Tr(\rho_{f,\fk{q}}(\Fr_p^r))$ from Theorem \ref{thm:hecke eigenvalues of a base-change lift} yields the desired result.

    In the case that $\rhoreg(\Fr_p) = 1$, that is, when $\Fr_p = 1$, we have $r = 1$ and $g = n$, so $p$ totally splits in $F$. For the Hecke eigenvalue of $h$ at $p$, we get $C(\fk{p}) = a(p)$, and the eigenvalues $\mu_j = 1$ for all $j = 1, \dots, n$. The computation from above now becomes
    \begin{align*}
        \prod_{i=1}^{2} \prod_{j=1}^{n}(1 - \lambda_i \mu_j) &= \prod_{i=1}^{2} \prod_{j=1}^{r}(1 - \lambda_i \mu_j x)^g \\
        &= \prod_{i=1}^{2}(1 - \lambda_i x)^g \\
        &= (1 - (\lambda_1 + \lambda_2)x + \lambda_1 \lambda_2 x^2)^g.
    \end{align*}
    Taking $g$-th roots gives the desired equalities $C(\fk{p}) = a(p)$ and $\psi(\fk{p}) = \chi(p)p^{k-1}$. This concludes the proof.
\end{proof}

\begin{remark}
    By using the decomposition of the regular representation $\rho$ into irreducible representations
    \begin{align*}
        \rhoreg = \bigoplus_{\rho \text{ irreducible}} \rho^{\dim \rho},
    \end{align*}
    and the additivity property of Artin $L$-functions, we see that
    \begin{align*}
        L(h,s) = L(\rho_{f,\fk{q}} \otimes \rhoreg, s) = \prod_{\rho \text{ irreducible}} L(\rho_{f, \fk{q}} \otimes \rho, s)^{\dim \rho}.
    \end{align*}
    This is the factorization mentioned in the introduction.
\end{remark}

\section{Base-change lifts of Hida families}
\label{sec: base-change lifts of Hida families}
We start by recalling the definition of \emph{$p$-ordinary classical modular forms}.
\begin{definition}
Let $f \in S_k(\Gamma_1(N), \chi)$ be a classical newform with Hecke eigenvalues $a_n(f)$. We say that \emph{$f$ is $p$-ordinary} if the Hecke characteristic polynomial at $p$,
\begin{align*}
    X^2 - a_p(f)X + \chi(p)p^{k-1},
\end{align*}
has at least one root that is a unit modulo $p$.
\end{definition}
Note that a modular form $f$ being $p$-ordinary is equivalent to requiring that the Hecke eigenvalue $a_p$ is a unit modulo $p$.

Fix a prime $p \nmid N$ and a $p$-ordinary Hecke newform $f \in S_k(\Gamma_1(N),\chi)$ with eigenvalues $a_n(f)$. Let $\alpha$ and $\beta$ denote the two roots of the Hecke characteristic polynomial of $f$ at $p$, chosen so that $\mathrm{ord}_p(\alpha) = 0$ and $\mathrm{ord}_p(\beta) = k-1$. 

\begin{definition} 
The \emph{$p$-ordinary stabilisation of a $p$-ordinary newform $f$} is the modular form whose $q$-expansion is given by
$$
f^{(p)}(q) := f(q) - \beta f(q^p).
$$
Notice that $f^{(p)}\in S_k(\Gamma_1(N)\cap\Gamma_0(p))$.
\end{definition}

Set $\Gamma = 1 + pN\mathbb{Z}_p$ and denote by $\Lambda = \mathbb{Z}_p[[\Gamma]]$ the completed group ring of $\Gamma$. Define the \emph{space of weights} as
$$
\Omega := \mathrm{Hom}_{\Z_p\text{-alg}}(\Lambda, \mathbb{C}_p)\cong \mathrm{Hom}_{\text{cts}}(\Gamma, \mathbb{C}_p^*).
$$
Define also the subset of \emph{classical characters of $\Omega$} as
$$
\Omega^{\cl} := \{\gamma \mapsto \gamma^k : k \in \mathbb{Z}_{\geq 2} \},
$$
where $\gamma = 1 + p$ is a topological generator of $\Gamma$.

For any finite flat extension $\Lambda'$ of $\Lambda$, let us define $\Omega' := \mathrm{Hom}(\Lambda',\O)$, endowed with a projection $\kappa:\Omega'\to\Omega$ induced by the ring inclusion $\Lambda\hookrightarrow\Lambda'$.

\begin{definition}
[Darmon--Rotger \cite{dr} p.~803]  A \emph{Hida family of tame level $N$} is a quad\-ruple $\mathbf{f}=(\Lambda_{\mathbf{f}},\Omega_{\mathbf{f}},\Omega_{\mathbf{f}}^{\cl},\mathbf{f}(q))$ such that
\begin{itemize}
\item[(a)] $\Lambda_{\mathbf{f}}$ is a finite flat extension of $\Lambda$,
\item[(b)] $\Omega_{\mathbf{f}}$ is a nonempty open subset in $X_{\mathbf{f}} := \mathrm{Hom}(\Lambda_{\mathbf{f}}, \mathbb{C}_p)$ and $\Omega_{\mathbf{f}}^{\cl}$ is a $p$-adically dense subset of $\Omega_{\mathbf{f}}$ whose image under $\kappa$ satisfies $\kappa(\Omega_{\mathbf{f}}^{\cl})\subseteq\Omega^{\cl}$, and
\item[(c)] $\displaystyle \mathbf{f} = \sum_{n\geq 1}\mathbf{a}_nq^n\in\Lambda_{\mathbf{f}} [[q]]$ is a formal $q$-series such that, for all $x\in\Omega_{\mathbf{f}}^{\cl}$, the weight $\kappa(x)$ specialization
$$
f^{(p)}_x:= \sum_{n=1}^{\infty}a_n(x)q^n
$$
is the $q$-expansion of the ordinary $p$-stabilization of a normalised newform $f_x$ of weight $\kappa(x)$ on $\Gamma_1(N)$. Notice that $f_x^{(p)}$ is a Hecke eigenform for $\Gamma_1(N)\cap\Gamma_0(p)$.
\end{itemize}
\label{hidafamily}
\end{definition}

As proved in \cite[Corollary 3.5]{hidaprevious}, such a Hida family is associated with the unique ring homomorphism $\lambda:h^{\operatorname{ord}}(N,\mathbb{Z}_p)\to \Lambda_f$, where $h^{\operatorname{ord}}(N,\mathbb{Z}_p)$ is the ordinary big Hecke algebra of level $N$ and $\lambda(T_n)=\mathbf{a}_n$ (see \cite[p.~297]{hidaprevious}). As we will justify next, Definition \ref{hidafamily} is equivalent to the following:

\begin{definition}
A \emph{Hida family of $p$-adic Galois representations} is a continuous Galois representation $\rho_{\lambda} : G_{\mathbb{Q}} \to \GL_2(\mathcal{K})$, where $\mathcal{K}$ is the field of fractions of a finite flat extension $\Lambda'$ of $\Lambda$ such that for each classical point $P_{k,\chi} \in \Spec(\Lambda')$, $\rho_\lambda \equiv \rho_{f_{k,\fk{p}}}\bmod{P_{k,\chi}}$ with $f_k\in S_k(N,\chi)$ a Hecke newform and $\fk{p} \in \mathcal{O}_{K_f}$ is a prime above $p$.
\label{hidareps}
\end{definition}

Indeed, due to Chebotarev Density Theorem, the set of Frobenius conjugacy classes is dense in $G_{\mathbb{Q}}$ with respect to the profinite topology. Hence, $\rho_{\lambda}$ is determined by its image at $\Fr_\ell$ for all primes $\ell \nmid pN$. In particular,
$$
\Tr(\rho_{\lambda}(\Fr_\ell)) \equiv a_\ell(f_k) \bmod{P_{k,\chi}}\mbox{ and } \det(\rho_{\lambda}(\Fr_\ell)) \equiv \chi(\ell)\ell^{k-1} \bmod{P_{k,\chi}}.
$$
Hence the elements $\Tr(\rho_{\lambda}(\Fr_\ell))=\mathbf{a}_\ell\in\Lambda'$ determine a Hida family in the sense of Definition \ref{hidafamily}.

One of the main features of Hida families is the following well-known result:
\begin{theorem}[Darmon--Rotger \cite{dr} p.~803] For any classical $p$-ordinary newform $f \in S_k(N, \chi)$, there exists a Hida family $(\Lambda_{\mathbf{f}},\Omega_{\mathbf{f}},\Omega_{\mathbf{f}}^{\cl},\mathbf{f}(q))$ of tame level $N$ which specializes to $f$ at weight $k$, namely, there exists a classical point $x_k\in\Omega_{\mathbf{f}}^{\cl}$ with $\kappa(x_k)=k$ such that $f_{x_k}= f$.
\end{theorem}

Next, we describe the constructions above generalized to Hilbert cusp forms. Let $F$ be a totally real number field and denote by $\fk{d}_F$ its different ideal. Let $h\in S_k(\fk{n}, \psi)$ be a normalized Hilbert newform over $F$ of parallel weight $k$. Thus, $T_{\fk{a}}h = C(\fk{a}) h$ for each integral ideal $\fk{a} \subseteq \O_F$. Let $K_h$ denote the number field generated by the set of Hecke eigenvalues of $h$ and denote by $\O_{K_h}$ its ring of integers. 
\begin{definition}
Let $p$ be a rational prime coprime with the level of $h$ and $\fk{p}$ a prime of $\O_F$ over $p$. We say that $h$ is \emph{nearly ordinary at $\fk{p}$} or \emph{$\fk{p}$-nearly ordinary} if $C(\fk{p})$ is a unit modulo $\fk{p}$. We say that $h$ is \emph{$p$-ordinary} if it is nearly ordinary at each prime $\fk{p} \subseteq \O_F$ above $p$.
\end{definition}

Like in the case of classical cusp forms, we have the following definition.
\begin{definition}
Let $h \in S_k(\fk{n}, \psi)$ be a Hilbert newform of parallel weight $k$, level $\fk{n}$ and nebentypus character $\psi$. Fix a prime $p$ coprime with $\fk{n}$ and let $\fk{p}$ be a prime of $\O_F$ over $p$, and suppose that $h$ is nearly ordinary at $\fk{p}$. Consider the root $\beta$ of the Hecke polynomial $X^2 - C(\fk{p})X + \psi(\fk{p})p^{k-1}$ such that $\beta \equiv 0 \bmod{\fk{p}}$. The \emph{ordinary $\fk{p}$-stabilization of $h$} is defined as
$$
h^{(\fk{p})} := h - \beta V_{\fk{p}}h,
$$
where the \emph{Verschiebung operator} $V_{\fk{p}}$ (the dual of Frobenius) is defined by
$$
V_{\fk{p}}h:=\sum_{\substack{\nu \in \O_F^* \\ \fk{p} \nmid \nu \\ \nu \gg 0}} a_{\nu}(h)q^{\nu}.
$$
If $h$ is $p$-ordinary, then the \emph{ordinary $p$-stabilization} is defined as in the case of classical modular forms as
$$
h^{(p)}(z) := h(z) - \beta h(pz),
$$
where $\beta$ is the non-unit root for the Hecke polynomial for $p$.
\end{definition}

Next, we fix finite a extension $\mathcal{O}$ of $\mathbb{Z}_p$. Let now $\Gamma := 1 + p^e\mathbb{Z}_p$, where $e:=[F\cap\mathbb{Q}_{\infty}:\mathbb{Q}]$ and $\mathbb{Q}_{\infty}$ is the cyclotomic $\mathbb{Z}_p$-extension of $\mathbb{Q}$. Let $u := (1+p)^{p^e}$ be a topological generator of $\Gamma$ and now set $\Lambda := \mathcal{O}[[T]]\cong \mathcal{O}[[\Gamma]]$, where $T$ is undetermined. Define again the \emph{space of weights} as
$$
\Omega := \mathrm{Hom}_{\O\text{-alg}}(\Lambda, \mathbb{C}_p)\cong \mathrm{Hom}_{\text{cts}}(\Gamma, \mathbb{C}_p^*).
$$
For $k \geq 2$ and a $p$-power root $\zeta$ of $1$, define the classical character $\mu_{k,\zeta}$ by setting $\mu_{k,\zeta}(1+T) := \zeta u^{k-2}$ and extending it by linearity to the full group ring. The subset of classical characters of $\Omega$ then becomes
$$
\Omega^{\cl} := \{\mu_{k,\zeta} :  k \geq 2, \ \zeta \ \mbox{is a $p$-power root of unity}\}.
$$

For any finite flat extension $\Lambda'$ of $\Lambda$, let us define $\Omega' := \mathrm{Hom}(\Lambda',\mathbb{C}_p)$, endowed with a projection $\kappa: \Omega' \to \Omega$ induced by the ring inclusion $\Lambda \hookrightarrow \Lambda'$.

The definition of Hida families generalizes to $\fk{p}$-nearly ordinary Hilbert modular forms:
\begin{definition}[Wiles \cite{wiles1988} p.~552]
A \emph{$\fk{p}$-nearly ordinary Hida family} of Hilbert cups forms of tame level $\fk{n}$ is a quadruple $\mathbf{h} = (\Lambda_{\mathbf{h}},\Omega_{\mathbf{h}},\Omega_{\mathbf{h}}^{\cl},\mathbf{h}(q))$ such that
\begin{itemize}
\item[(a)] $\Lambda_{\mathbf{h}}$ is a finite flat extension of $\O_{F,\fk{p}}[[T]]$,
\item[(b)] $\Omega_{\mathbf{h}}\neq \emptyset$ is an open subset of $\Omega$, 
\item[(c)] $\Omega_{\mathbf{h}}^{\cl}\subseteq \Omega_{\mathbf{h}}\cap\kappa^{-1}(\Omega^{\cl})$ is a dense subset in $\Omega_{\mathbf{h}}$,
\item[(d)] $\displaystyle \mathbf{h}= \sum_{\substack{\nu \in \fk{d}_F \\
\nu \gg 0}} \mathbf{a}_{\nu} q^{\nu} \in \Lambda_{\mathbf{h}} [[q^{\nu}]]$ is a formal $q$-series such that for all $\mu_{k,\zeta}\in \Omega_{\mathbf{h}}^{\cl}$ the specialization
$$
h_k^{(\fk{p})}(q) := \sum_{\substack{\nu \in \fk{d}_F \\
\nu \gg 0}} \mu_{k,\zeta}(a_{\nu})q^{\nu}
$$
is the $q$-expansion of a $\fk{p}$-nearly ordinary Hilbert modular form $h_k$ of parallel weight $k$ and level $\fk{n}$ defined over $\mathcal{O}[\zeta]$.
\end{itemize}
\label{hidafamilyhilbert}
\end{definition}

Since we will only deal with the case $\zeta = 1$, we make the simplifying assumption that $\zeta = 1$.

For $k \geq 2$, let $\tilde{P}_{k}$ be a prime ideal of $\Lambda'$ over $P = (T + 1 - u^{k-2})$. We have that
\begin{lemma}[Wiles \cite{wiles1988} p.~545] Condition (d) in Definition \ref{hidafamilyhilbert} is equivalent to $\mathcal{O}_{K_{h_k}} \subseteq \Lambda'/\tilde{P}_k$ and $\mathbf{h}\equiv h_k^{(\fk{p})} \bmod{\tilde{P}_k}$.
\end{lemma}

Again, by \cite[Corollary 2.5]{hida}, a Hida family of $\fk{p}$-nearly ordinary Hilbert modular forms is determined by a ring homomorphism $\lambda : h^{n,\operatorname{ord}}(\fk{n},\O_{F,\fk{p}}) \to \Lambda'$, where $h^{n,\operatorname{ord}}(\fk{n},\O_{F,\fk{p}})$ is the nearly ordinary Hecke algebra of level $\fk{n}$ and $\lambda(T_{\nu}) = \mathbf{a}_{\nu}$ (see \cite[p.~150]{hida}). Thus, taking again into account that the Frobenius elements $\{\Fr_{\fk{L}} : \fk{L} \nmid \fk{p}\fk{n}\}$ are dense in $G_{F}$, Definition \ref{hidafamilyhilbert} is equivalent to:

\begin{definition} 
Given a totally real number field $F$ and a prime ideal $\fk{p}$ of $\O_F$, \emph{a Hida family of $\fk{p}$-adic Galois representations of $G_F$} is a continuous Galois representation $\rho_{\lambda}: G_F \to \GL_2(\mathcal{K})$ where $\mathcal{K}$ is the field of fractions of a finite flat extension $\Lambda'$ of $\mathcal{O}_{F,\fk{p}}[[\Gamma]]$ such that for each classical point $P_{k,\psi} \in \Spec(\Lambda')$, we have $\rho_{\lambda} \equiv \rho_{h_{k,\fk{P}}} \bmod{P_{k,\psi}}$ with $h_k\in S_{k}(\fk{n},\psi)$ a Hilbert cuspidal newform of parallel weight $k$ and $\fk{P}$ a prime ideal of $\Lambda' / P_{k,\psi}$ over $\fk{p}$.
\end{definition}

We can recover the Fourier coefficients at primes $\nu \nmid \fk{pn}$ as
$$
\mathbf{a}_{\nu} := \Tr(\rho_{\lambda}(\Fr_{\nu})).
$$

\begin{theorem}[Wiles \cite{wiles1988} Theorem 1.4.1, Hida \cite{hida} Theorem 2.4] For any normalized Hilbert newform $h \in S_{k}(\fk{n},\psi)$ over $F$ of parallel weight $k \geq 2$, there exists a nearly ordinary Hida family $(\Lambda_{\mathbf{h}}, \Omega_{\mathbf{h}}, \Omega_{\mathbf{h}}^{\cl}, \mathbf{h}(q))$ such that $h_{k}=h$.
\end{theorem}

We propose the following definition for the base-change lift of a Hida family.
\begin{definition}
\label{basechangehida}
Let $F$ be a totally real field and $\rho_{\lambda}: G_{\mathbb{Q}} \to \GL_2(\mathcal{K})$ a Hida family of $p$-adic Galois representations of tame level $N$. Let $\fk{p}$ a prime ideal of $F$ above $p$ coprime to $N$ and $\fk{n}$ an ideal of $F$ above $N$. A \emph{base-change lift of $\rho_{\lambda}$ to $F$} is a Hida family of $\fk{p}$-adic Galois representations ${\rho_{\lambda,h}}: G_F \to \GL_2(\mathcal{K})$ such that for any prime $\fk{L}$ of $F$ above an unramified prime $\ell \nmid pN$, we have
\begin{align*}
    \rho_{\lambda}|_{G_F}(\Fr_{\fk{L}})={\rho_{\fk{\lambda},h}}(\Fr_{\fk{L}}).
\end{align*}
\end{definition}
At the level of formal $q$-expansions, Definition \ref{basechangehida} is equivalent to the following:

\begin{definition}
Let $\mathbf{h} = (\Lambda_{\mathbf{h}},\Omega_{\mathbf{h}},\Omega_{\mathbf{h}}^{\cl},\mathbf{h}(q))$ be a $\fk{p}$-nearly ordinary Hida family of Hilbert modular forms of tame level $\fk{n}$ and $\mathbf{f}=(\Lambda_{\mathbf{f}},\Omega_{\mathbf{f}},\Omega_{\mathbf{f}}^{\cl},\mathbf{f}(q))$ a Hida family of tame level $N$ with $\fk{n} \mid N$. We say that \emph{$\mathbf{h}$ is a base-change lift of $\mathbf{f}$ to $F$} if for any prime $\fk{L}$ of $F$ above an unramified prime $\ell \nmid pN$, we have
$$
\Tr(\rho_{f_x,\fk{p}}|_{D_{\fk{L}}}(\Fr_\fk{L}|_{D_{\fk{L}}})) = \Tr(\rho_{h_x,\fk{P}}(\Fr_{\fk{L}})) = \mu_{k,1}(\mathbf{a}_{\fk{L}}) = C(\fk{L},h_x),
$$
where $\fk{P}$ is a prime in $\mathcal{O}_{K_{h_x}}$ above $\fk{p}$.
\end{definition}

Now we can prove the main result of this section.
\begin{theorem} 
Let $F$ be a totally real field. Then every Hida family admits a base-change lift to $F$.
\end{theorem}
\begin{proof} 
Let us consider a Hida family $\mathbf{f} = (\Lambda_{\mathbf{f}}, \Omega_{\mathbf{f}}, \Omega_{\mathbf{f}}^{\cl}, \mathbf{f}(q))$ of tame level $N$ with $\mathbf{f}= \sum_{n\geq 1} \mathbf{a}_n q^n$ attached to a representation $\rho_{\lambda}:\mathrm{G}_{\mathbb{Q}}\to\mathrm{GL}_2(\mathcal{K})$, where $\mathcal{K}$ is the field of fractions of some finite flat extension of $\Lambda$ determined by the evaluations $\rho_{\lambda}|_{D_{\ell}}(\Fr_\ell)$ at primes $\ell \nmid pN$.

Now, for a prime $\fk{L}$ in $F$ over a rational prime $\ell \nmid pN$, let us define the element $\mathbf{C}_{\fk{L}} = \mathbf{C}_{\fk{L}}(x) \in \Lambda_{\mathbf{f}}$ as
\begin{align*}
        \mathbf{C}_{\fk{L}} = \begin{cases}
            \mathbf{a}_\ell & \text{if } r = 1 \\
            \mathbf{a}_{\ell^r} - \chi(\ell)\ell^{x-1} \mathbf{a}_{\ell^{r-2}} & \text{otherwise},
        \end{cases}
    \end{align*}
where $r = f(\fk{L} \mid \ell)$ is the residual degree of $\fk{L}$ over $\ell$. This element is defined over the set of classical characters and extended to the whole weight space by density.

Notice that $\mathbf{C}_\fk{L} = \Tr(\rho_{\lambda}|_{D_{\fk{L}}}(\Fr_{\fk{L}}))$ since for every classical weight $x$ we have that
$$
\mathbf{C}_{\fk{L}}(x) = \Tr(\rho_{\lambda}|_{D_{\fk{L}}}(\Fr_\fk{L}|_{D_{\fk{L}}}))(x) = \Tr(\rho_{f_x,\fk{p}}|_{D_{\fk{L}}}(\Fr_\fk{L}|_{D_{\fk{L}}})) ,
$$
due to Theorem \ref{thm:hecke eigenvalues of a base-change lift} applied to the weight $x$ specialization of $\mathbf{f}$. The specialization to classical weights is enough because the set of classical weights is dense in the full space of weights. Observe that for each $\fk{p}$ of $F$ over $p \nmid N$ and for each $x \in \Omega_{\mathbf{f}}^{\cl}$, the $p$-adic representation $\rho_{f_x,\fk{p}}$ is fully determined by its evaluation at the Frobenius elements $\Fr_\ell$ for $\ell \nmid pN$. Likewise, $\rho_{f_x,\fk{p}}$ admits a base-change lift $\rho_{h_x,\fk{P}}$ and it is also fully determined by its evaluations at Frobenius elements $\Fr_\fk{L}$ at primes $\fk{L}$ of $F$ over unramified primes $\ell \nmid pN$.

This leads us to define $\rho_{\lambda, h}$ as the unique continuous Galois representation defined, up to conjugation, by setting for each prime $\fk{L} \nmid \fk{pn}$ over $\ell$
$$
\Tr({\rho_{\lambda,h}}|_{D_{\fk{L}}}(\Fr_{\fk{L}})) := \mathbf{C}_{\fk{L}}
$$
and
$$
\det({\rho_{\lambda,h}}|_{D_{\fk{L}}}(\Fr_{\fk{L}})) := \chi(\ell^r)\ell^{r(k-1)}.
$$
By construction, we see that, for classical weights $x$
\begin{align*}
    \mathbf{C}_{\fk{L}}(x) &= \Tr(\rho_{\lambda, h}(\Fr_{\fk{L}}))(x) \\
    &= \Tr(\rho_{h_x,\fk{P}}(\Fr_{\fk{L}})) \\
    &= \Tr(\rho_{f_x,\fk{p}}|_{D_{\fk{L}}}(\Fr_\fk{L}|_{D_{\fk{L}}})) \\
    &= \Tr(\rho_{\lambda}|_{D_{\fk{L}}}(\Fr_\fk{L}|_{D_{\fk{L}}}))(x),
\end{align*}
and a similar equality holds for the determinants. Since the classical weights $x$ form a dense subset, we deduce that
\begin{align*}
    \Tr(\rho_{\lambda, h}|_{D_{\fk{L}}}(\Fr_{\fk{L}})) &=
    \Tr(\rho_{\lambda}|_{D_{\fk{L}}}(\Fr_\fk{L}|_{D_{\fk{L}}}))
\end{align*}
and
\begin{align*}
    \det(\rho_{\lambda, h}|_{D_{\fk{L}}}(\Fr_{\fk{L}})) &=
    \det(\rho_{\lambda}|_{D_{\fk{L}}}(\Fr_\fk{L}|_{D_{\fk{L}}})).
\end{align*}
Hence, up to conjugation, the representations satisfy
$$
\rho_{f_x,\fk{p}}|_{D_{\fk{L}}}(\Fr_\fk{L}|_{D_{\fk{L}}}) = \rho_{\lambda}|_{D_{\fk{L}}}(\Fr_{\fk{L}})(x).
$$
We conclude that $\rho_{\lambda, h}$ is a base-change lift of $\rho_{\lambda}$ to $F$.
\end{proof}

\subsection{An application: potentially diagonalizable automorphic lifts of large weights}

Denote by $\mathcal{O}_{\overline{\mathbb{Q}}_p}$ the closed unit ball in $\overline{\mathbb{Q}}_p$.  The following equivalence relation is introduced in \cite[p.~530]{blggt}: For a local finite extension $F$ of $\Q_p$, let $\rho_1, \,\rho_2: G_F \to \GL_n(\mathcal{O}_{\overline{\Q}_p})$ be two Galois representation. We say that $\rho_1$ \emph{connects to} $\rho_2$ if
\begin{itemize}
\item $\overline{\rho}_1$ and $\overline{\rho}_2$ are equivalent,
\item $\rho_1$ and $\rho_2$ are potentially crystalline,
\item for each continuous field embedding $\tau: F \to \overline{\Q}_p$, $\operatorname{HT}_{\tau}(\rho_1) = \operatorname{HT}_{\tau}(\rho_2)$, namely, the representations have the same set of Hodge--Tate weights,
\item both representations define points on the same irreducible component of the framed universal deformation ring that parametrizes deformations of the common residual representation which are potentially crystalline.
\end{itemize}

\begin{definition}[BLGGT \cite{blggt} p.~531] 
We call a Galois representation $\rho: G_F \to \GL_n(\mathcal{O}_{\overline{\Q}_p})$ \emph{diagonalizable} if it is crystalline and connects to some representation $\chi_1\oplus \dots \oplus \chi_n$, where $\chi_i$ are crystalline characters. The representation $\rho$ is said to be \emph{potentially diagonalizable} if there exists a finite extension $F'/F$ such that $\rho|_{G_{F'}}$ is diagonalizable.
\end{definition}

The following result is proved in \cite[Theorem 2.6]{blancodieulefait1}.
\begin{theorem}
Let $f \in S_k(N,\chi)$ be a cuspidal Hecke eigenform. Fix a prime $p > \max\{k,6\}$ and a prime $\fk{p}$ of $\mathcal{O}_{K_f}$ above $p \nmid N$. Then, there exists a sequence $\{k_r\}_{r\geq 1} \subseteq \mathbb{N}$ such that $\overline{\rho}_{f,\fk{p}}$ admits a potentially diagonalizable modular lift of Hodge--Tate weights $\{0,k_r-1\}$. In fact, if $f$ is $p$-ordinary, the condition $p > \max\{k,6\}$ can be lifted.
\label{podi}
\end{theorem}

We can now prove a generalization of Theorem \ref{podi} in the ordinary case:
\begin{theorem}
Let $f\in S_k(N,\chi)$ be a $p$-ordinary cuspidal Hecke eigenform, $\mathcal{O}_{K_f}$ the ring of integers of its field of definition and $\fk{p}$ a prime of $\mathcal{O}_{K_f}$ above $p \nmid N$. Then, for each totally real number field $F$, there exists a sequence $\{k_r\}_{r\geq 1} \subseteq \mathbb{N}$ such that $\overline{\rho}_{f,\fk{p}}|_{G_F}$ admits a potentially diagonalizable automorphic lift of Hodge--Tate weights $\{0,k_r-1\}$.    
\end{theorem}

\begin{proof}
For $F=\mathbb{Q}$ this was proved in \cite[Theorem 2.6]{blancodieulefait1} by considering a Hida family $\mathbf{f}$ passing by $f$ at weight $k$, this means that $\rho_{f_k,\fk{p}} \cong \rho_{f,\fk{p}}$. In particular, there exists a sequence $\{k_r\}_{r\geq 1} \in \mathbb{N}$ such that $\rho_{f_{k_r},\fk{p}_r}$ is a potentially diagonalizable modular lift of $\overline{\rho}_{f,\fk{p}}$ of Hodge--Tate weights $\{0,k_r-1\}$, where $\fk{p}_r$ is a prime of $\mathcal{O}_{K_{f_{k_r}}}$ above $p$. The key idea to establish this fact is to observe that 
\begin{itemize}
    \item $\rho_{f_{k_r},\fk{p}_r}$ is ordinary because $f_{k_r}$ is ordinary, and
    \item $\rho_{f_{k_r},\fk{p}_r}$ is crystalline since $p\nmid N$,
\end{itemize}
and to invoke \cite[Lemma 1.4.3]{blggt}.

Now, for a totally real number field $F$, let us consider the base-change lift $\mathbf{h}$ of $\mathbf{f}$ to $F$, which passes by the base-change lift $h$ of $f$ to $F$ at weight $k$, that is,
$$
\rho_{h_k,\fk{p}} \cong \rho_{f,\fk{p}}|_{G_F}.
$$

Consider the sequence of weights $\{k_r\}_{r\geq 1}$ provided by \cite[Theorem 2.6]{blancodieulefait1}, and the corresponding specializations $f_{k_r}$ of $\mathbf{f}$ at these weights. By construction, the specializations $h_{k_r}$ of $\mathbf{h}$ at these weights are the base-change lifts of the $f_{k_r}$ to $F$, namely:
$$
\rho_{h_{k_r},\fk{p}_r} \cong \rho_{f_{k_r},\fk{p}_r}|_{G_F}.
$$

Now, for any Hecke eigenform $g \in S_k(N,\chi)$ and any prime $\ell \nmid N$, denoting by $a_{\ell^n}$ its $\ell^n$-th Hecke eigenvalue, it is well known that
$$
a_{\ell^n}(g) = a_\ell(g) a_{\ell^{n-1}}(g) - \ell^{k-1} \chi(\ell) a_{\ell^{n-2}}(g),
$$
from which, by induction, it is straightforward to see that 
\begin{equation}
    a_{\ell^n}(g)\equiv a_\ell(g) \mod{\ell}.
    \label{Heckecoefsord}
\end{equation}
Now, $f_{k_r}$ are $p$-ordinary, and hence $a_{p}(f_{k_r})$ is a $p$-adic unit.

But if $p$ is unramified in $F$, then from Equation \eqref{Heckecoefsord} and Theorem \ref{thm:hecke eigenvalues of a base-change lift}, we see that for each prime ideal $\fk{p}$ of $F$ over $p$, the $\fk{p}$-th Hecke eigenvalue of $h_{k_r}$ is also a $p$-adic unit. Hence, $h_{k_r}$ is $\fk{p}$-nearly ordinary and consequently $p$-ordinary. All this implies that $\rho_{h_{k_r},\fk{p}_r}$ is ordinary.

Likewise, since $p \nmid N$, a fortiori $\fk{p} \nmid \fk{n}$ for any $\fk{p}$ above $p$. Therefore, $\rho_{h_{k_r},\fk{p}_r}$ is crystalline. Thus, again by \cite[Lemma 1.4.3]{blggt}, $\rho_{h_{k_r},\fk{p}_r}\cong\rho_{f_{k_r},\fk{p}_r}|_{G_F}$ is potentially diagonalizable.
\end{proof}

As for the non-ordinary case, we need to impose, as in \ref{podi} that $p>\max\{k,6\}$.
\begin{theorem}
Let $f\in S_k(N,\chi)$ be a non-ordinary cuspidal Hecke eigenform, $\mathcal{O}_{K_f}$ the ring of integers of its field of definition and $\fk{p}$ a prime of $\mathcal{O}_{K_f}$ above $p \nmid N$. Assume that $\overline{\rho}_{f,\fk{p}}|_{G_F}$ has large image, $f$ is not a CM form, and $p > \max\{k,6\}$. Then, for each totally real number field $F$, there exists a sequence $\{k_r\}_{r\geq 1} \subseteq \mathbb{N}$ such that $\overline{\rho}_{f,\fk{p}}|_{G_F}$ admits a potentially diagonalizable automorphic lift of Hodge--Tate weights $\{0,k_r-1\}$.  
\end{theorem}

\begin{proof} Again, the case $F=\mathbb{Q}$ was proved in \cite[Theorem 0.1]{blancodieulefait1}. They used a method due to Khare and Winterberger which allows to produce a global deformation ring parametrizing deformations of $\overline{\rho}_{f,\fk{p}}$ with the desired local properties, except for the fact that these deformations are not necessarily automorphic. So, using a solvable base change and \cite[Theorem 4.3.1]{blggt} one obtains the desired potentially diagonalizable modular deformation of Hodge--Tate weights $\{0,k_r-1\}$ for $k_r$ in an infinite family of integers.

Consider the Galois representation $\rho_{f_r,\fk{p}_r}$ residually equivalent to $\overline{\rho}_{f,\fk{p}}$, where $f_r$ has weight $k_r$ and $\fk{p}_r$ is a prime ideal of $\mathcal{O}_{K_{f_r}}$ above $p$. The condition of being potentially diagonalizable is compatible with base change, so $\rho_{f_r,\fk{p}_r}|_{G_F}$ is potentially diagonalizable. Also, by \ref{thm: existance of base-change lifts to F} $\rho_{f_r,\fk{p}_r}|_{G_F}$ is automorphic, hence attached to a Hilbert cusp form of parallel weight $k_r$ so that $\rho_{f_r,\fk{p}_r}|_{G_F}$ has Hodge--Tate weights $\{0,k_r-1\}$.
\end{proof}

Our motivation to define the base change of Hida families as well as the non-ordinary potentially diagonalizable families of modular Galois representations is to tackle the Langlands base change of tensor products and symmetric powers of modular automorphic representations. In \cite{newton2021symmetric}, Hida families have been used to establish the automorphy of symmetric powers of automorphic representations attached to $\mathrm{GL}_2/\mathbb{Q}$, hence we expect that the base change of Hida families can be exploited to establish the automorphic base change of symmetric powers to $\mathrm{GL}_2/F$ for totally real number fields $F$. 

On the other hand, \cite{arias2016automorphy} uses potentially diagonalizable families of variable weights in conjunction with safe chains of modular Galois representations to prove that the tensor product of a significant family of automorphic representations attached to $\mathrm{GL}_2/\mathbb{Q}$ stays automorphic. Hence, we plan to combine base change with these results to establish the automorphic base change of tensor products to $\mathrm{GL}_2/F$. We will address these two problems in future research.

\section{Implementation in Magma}
\label{sec: implementation in Magma}
This section provides the pseudocode for implementing the formulas from Theorem \ref{thm:hecke eigenvalues of a base-change lift} in Magma \cite{Magma1484478}. Furthermore, we use this formula to prune the search for a Hilbert modular form lifting some classical newform $f$. The Magma code implementations of the introduced algorithms can be found in Appendix \ref{appendix: Magma implementations}.

\subsection{Hecke Eigenvalues of a base-change lift}
We have implemented an algorithm in Magma that computes the Hecke eigenvalues of a base-change lift of a newform $f \in S_k(\Gamma_1(N), \chi)$ to a totally real number field $F$. The pseudocode is given below in Algorithm \ref{alg:HeckeEigenvalueOfBasechangeLift}.

\begin{algorithm}
\caption{Hecke eigenvalue of base-change lift}
\label{alg:HeckeEigenvalueOfBasechangeLift}
\SetKw{eq}{eq}
\SetKwInOut{Input}{Input}
\SetKwInOut{Output}{Output}
\Input{A newform $f \in S_k(\Gamma_1(N), \chi)^{\operatorname{new}}$ with Fourier coefficients $[a(n)]_{n = 1}^\infty$, \\
a prime ideal $\fk{p} \subset \O_F$}
\Output{$C(\fk p) =$ Hecke Eigenvalue at $\fk{p}$ of the base-change lift $f$ to $F$}
\BlankLine
$k := $ Weight($f$)\;
$p := \fk{p} \cap \Z$, the rational prime below $\fk{p}$\;
$r := f(\fk{p} \mid p)$, residual degree\;
\eIf{$r$ \eq $1$}{
    $C(\fk{p}) := a(p)$\;
}
{
    $C(\fk{p}) := a(p^{r}) - \chi(p) p^{k-1} a(p^{r-2})$\;
}
\Return{$C(\fk{p})$\;}
\end{algorithm}

\subsection{Computing the base-change lift}
We can use Algorithm \ref{alg:HeckeEigenvalueOfBasechangeLift} to try to pinpoint the lifted Hilbert modular form that lies above some newform $f \in S_k(\Gamma_1(N), \chi)$. However, due to the current limitations of the Magma package for Hilbert modular forms, the Magma implementation can only compute the space of Hilbert newforms for trivial nebentypus $\chi = 1$. Since it possible for the base-change lift $h$ to lose some of the ramification of the newform $f$, the level $\fk n$ where $h$ is new could be a priori any divisor of the ideal $N \O_F$. Hence, the following algorithm loops through all possible levels, that is, ideals dividing $N \O_F$, and computes the new subspace for that level. The pseudo-code for finding the base-change lift is the following.

\begin{algorithm}[H]
\caption{Base-change lift to $F$}
\label{alg:BasechangeLiftToF}
\SetKw{eq}{eq}
\SetKw{noteq}{neq}
\SetKwInOut{Input}{Input}
\SetKwInOut{Output}{Output}
\Input{A newform $f \in S_k(\Gamma_1(N), \chi)^{\operatorname{new}}$, \\
a totally real number field $F$, \\
NormUpperbound giving the bound on the norm of the primes}
\Output{A base-change lift $h$ of $f$ to $F$}
\BlankLine
$N := $ Level($f$)\; 
$D_F := $ Discriminant($F$)\;
possibleLevels $:=$ $\{\fk{a} \subset \O_F : \fk{a} \mid (N \O_F)\}$\;
goodPrimes $:= \{\fk{p} \subset \O_F : \fk{p} \nmid (D_FN)\O_F \text{ prime}, \ \Norm_{F / \Q}(\fk{p}) \leq \text{NormUpperbound}\}$\;

possibleLifts := EmptyList\;

\For{level in possibleLevels}
{
    eigenforms $:=$ HilbertNewforms($F$, level, $k$)\;
    \For{$h$ in eigenforms}
    {
        flag $:=$ true\;
        \For{$\fk{p}$ in goodPrimes}
        {
            $C(\fk{p}) := $ HeckeEigenvalueOfBasechangeLift($f$, $\fk{p}$)\;
            $A(\fk{p}) := $ HeckeEigenvalue($h$, $\fk{p}$)\;
            \If{$A(\fk p)$ \noteq $C(\fk {p})$}
            {
                flag $:=$ false\;
                break\;
            }
        }
        \If{flag}
        {
            Append(possibleLifts, h)\;
        }
    }
}

\eIf{\#possibleLifts \eq $1$}
{
    \Return{possibleLifts[1]}\;
} 
{
    \Return{Error: Found more than one possible lift.}
}
\end{algorithm}

\subsection{Examples}

\begin{example}
The example \href{https://magma.maths.usyd.edu.au/magma/handbook/text/1789}{ModFrmHil\_eigenform-examples (H149E6)} in the Magma documentation computes the Hecke eigenvalues of the one-dimensional piece of the newforms coming from $F = \Q(\sqrt{2})$ and level $\fk{n} = 11 \O_F$. However, the example compares the Hecke eigenvalues of $h$ with those of an classical cusp form $f \in S_2(\Gamma_0(11))$ only at split primes $\fk{p} \in \Spec(\O_F)$. We can now complete the example at the inert primes 3 and 5. 
\begin{verbatim}
_<x> := PolynomialRing(IntegerRing());
F := NumberField(x^2-2);  
M := HilbertCuspForms(F, 11*Integers(F));
decomp := NewformDecomposition(NewSubspace(M)); 

h := Eigenform(decomp[1]);  
f := Newforms(CuspForms(11))[1][1];

primes := [P : P in PrimesUpTo(50,F) | InertiaDegree(P) eq 2];
for P in primes do
    Cp := HeckeEigenvalueOfBasechangeLift(f, P);
    Norm(P), HeckeEigenvalue(h,P), Cp;
end for;

> 9  -5 -5
> 25 -9 -9
\end{verbatim}
\end{example}

\begin{example}
Let $F = \Q(\zeta_7)^+ = \Q(\zeta_7) \cap \R$ be the maximal totally real subextension of the 7-th cyclotomic field $\Q(\zeta_7)$. Then $F$ is a cubic abelian totally real number field. Consider the newform $f \in S_2(\Gamma_0(N))$, where $N = 147 = 3 \cdot 7^2$ with LMFDB label \href{https://www.lmfdb.org/ModularForm/GL2/Q/holomorphic/147/2/a/c/}{147.2.a.c}. By computing the Hecke eigenvalues of a lift  of $f$ to $F$, we see that they agree with the Hilbert modular form $h$ of prime level $\fk{n} = 3 \O_F$ with the LMFDB label \href{https://www.lmfdb.org/ModularForm/GL2/TotallyReal/3.3.49.1/holomorphic/3.3.49.1-27.1-a}{3.3.49.1-27.1-a}. We note that we have lost the ramifying prime $p = 7$ in the level of the lift.

\begin{verbatim}
_<zeta> := CyclotomicField(7);
F := NumberField(MinimalPolynomial(zeta + 1/zeta));

H := HilbertCuspForms(F, 3*Integers(F));
Hdecomposed := NewformDecomposition(NewSubspace(H));
h := Eigenform(Hdecomposed[1]);
f := Newforms(CuspForms(147))[3][1];

primes := [P : P in PrimesUpTo(50,F) 
| (AbsoluteDiscriminant(F) mod Norm(P)) ne 0];
for P in primes do
    Cp := HeckeEigenvalueOfBasechangeLift(f, P);
    Norm(P), HeckeEigenvalue(h,P), Cp;
end for;

> 8 -4 -4
> 13 1 1
> 27 1 1
> 29 4 4
> 41 -10 -10
> 43 5 5
\end{verbatim}
\end{example}

\begin{example}
As an example of a prime $p$ with a high inertia degree, we consider the totally real field $F = \Q(\zeta_{11})^+$, the classical modular form $f \in S_2(\Gamma_0(11))^{\operatorname{new}}$ with the LMFDB label \href{https://www.lmfdb.org/ModularForm/GL2/Q/holomorphic/11/2/a/a/}{11.2.a.a}. The Hilbert modular form $h$ with the LMFDB label \href{https://www.lmfdb.org/ModularForm/GL2/TotallyReal/5.5.14641.1/holomorphic/5.5.14641.1-11.1-a}{5.5.14641.1-11.1-a} is the lift of $f$ to $\mathbb{Q}(\zeta_{11})^+$. The level $\fk{n}$ of $h$ is the unique prime above the totally ramified prime 11. The prime 2 stays inert in $F$, so at $\mathfrak{p} = 2\mathcal{O}_F$, we have residual degree $r = [F : \mathbb{Q}] = 5$, and our formula computes the following Hecke eigenvalue $C(\fk{p})$ for the lift of $f$ to $F$.
\begin{verbatim}
_<zeta> := CyclotomicField(11);
F := NumberField(MinimalPolynomial(zeta + 1/zeta));
OF := Integers(F);

level := Factorisation(11*OF)[1][1];
H := HilbertCuspForms(F, level);
Hdecomposed := NewformDecomposition(NewSubspace(H));
h := Eigenform(Hdecomposed[1]);
f := Newforms(CuspForms(11))[1][1];

P := 2*OF;
Cp := HeckeEigenvalueOfBasechangeLift(f, P);
"r =", InertiaDegree(P);
HeckeEigenvalue(h, P), Cp;

> r = 5
> 8 8
\end{verbatim}
\end{example}

\section*{Statements and declarations}
I. Blanco-Chacón and L. Dieulefait have been supported by PID2022-136944NB-I00 (Spanish Ministry of Science and Innovation). A. Haavikko has been supported by Wihuri foundation (grant \#00240063).

\printbibliography 

\appendix

\section{Magma implementations}
\label{appendix: Magma implementations}
\subsection{Algorithm \ref{alg:HeckeEigenvalueOfBasechangeLift}}
\begin{verbatim}
--- 
HeckeEigenvalueOfBasechangeLift(f,P)

**Inputs:**
- f, a newform of weight k and level N for the congruence subgroup Gamma0(N)
- P, a prime ideal of O_F not dividing the level N, where
  F is totally real number field into which we base-change

**Output:** 
- Hecke eigenvalue at P above p of the Hilbert modular form h
  which is a base-change lift of f to F.

HeckeEigenvalueOfBasechangeLift := function(f, P)
    r := InertiaDegree(P);
    _, p := IsPrincipal(P meet IntegerRing());
    k := Weight(f);
    chi := DirichletCharacter(f);
    if r eq 1 then
            Cp := Coefficient(f, p);
    else
            Cp := Coefficient(f, p^r) - chi(p)*p^(k-1)*Coefficient(f, p^(r - 2));
    end if;
    return Cp;
end function;
---
\end{verbatim}

\subsection{Algorithm \ref{alg:BasechangeLiftToF}}
We have written several auxiliary functions that simplify and improve the readability of the implementation of Algorithm \ref{alg:BasechangeLiftToF}.
\begin{verbatim}
---
PossibleLevelsForBasechangeLift(f, F)

**Inputs**
- f, a Hecke newform of weight k and level N
  for the congruence subgroup Gamma0(N)
- F, totally real number field into which we base change

**Output** 
- List of possible levels for the base-change HMF h of f to F.

PossibleLevelsForBasechangeLift := function(f, F);
    // Loop through all divisors of N * OF counting with multiplicity
    N := Level(f);
    OF := Integers(F);

    prime_multiset := [* *];
    for x in Factorisation(N*OF) do
        for i in [1..x[2]] do
            Append(~prime_multiset, x[1]);
        end for;
    end for;

    subset_indices := Subsets(Set([1..#prime_multiset]));
    possible_levels := { };
    for index in subset_indices do
        a := 1*OF;
        for j in index do
            a := a*prime_multiset[j];
        end for;
        Include(~possible_levels, a);
    end for;

    return [level : level in possible_levels];
end function;
---


---
 GoodPrimes(f, F, NormUpperbound)
 
**Inputs**
- f, a Hecke newform of weight k and level N
  for the congruence subgroup Gamma0(N)
- F, totally real number field into which we base change

**Output** 
- List of primes P of O_F
  that are above unramified primes p that do not divide N

GoodPrimes := function(f, F, NormUpperbound)
    N := Level(f);
    D := AbsoluteDiscriminant(F);
    good_primes := [* *];

    for P in PrimesUpTo(NormUpperbound, F) do
        _, p := IsPrincipal(P meet IntegerRing());
        if N*D mod p ne 0 then
            Append(~good_primes, P);
        end if;
    end for;

    return good_primes;
end function;
---


---
IsBasechangeLift(h, f, primes, good_primes, embedding, automorphisms_Kf)

**Inputs**
- h, a Hilbert newform of parallel weight k and level dividing N * O_F
  for a totally real field F.
- f, a Hecke newform of weight k and level N for the congruence 
  subgroup Gamma0(N)
- good_primes, a list of prime ideals of F above unramified primes 
  that do not divide N
- embedding, any embedding of the Hecke eigenvalue field K_f of f into that of h
- automorphisms_Kf, field automorphisms of the Hecke eigenvalue field K_f

**Output** 
- Boolean true/false based on whether h has the Hecke eigenvalues 
  of a basechange lift of f to F at all P in good_primes
  
IsBasechangeLift := function(h, f, good_primes, embedding, automorphisms_Kf)
    assert forall{kh : kh in Weight(Parent(h)) | kh eq Weight(f)};

    flag := true;
    for P in good_primes do
        Cp := HeckeEigenvalueOfBasechangeLift(f, P);
        Ap := HeckeEigenvalue(h, P);
        Cp_embeddings_to_Kh := [embedding(map(Cp)) : map in automorphisms_Kf];
        if Ap notin Cp_embeddings_to_Kh then
            flag := false;
            break;
        end if;
    end for;

    return flag;
end function;
---


---
BasechangeLift(f, F, NormUpperbound, ReturnFirstPossibleLift)

**Inputs**
- f, a Hecke newform of weight k and level N for 
  the congruence subgroup Gamma0(N) 
- F, totally real number field into which we base change

**Output** 
- List of potential Hilbert modular forms h that share the first 
  Hecke eigenvalues with the base-change lift of f to F.

BasechangeLift := function(f, F : NormUpperbound := 100, ReturnFirstPossibleLift := false)
    assert IsTotallyReal(F);
    assert IsTrivial(DirichletCharacter(f)); 
    // Magma currently only supports HMFs with trivial DirichletCharacter

    k := Weight(f);
    n := Degree(F);
    Kf := CoefficientField(f);
    automorphisms_Kf := Automorphisms(Kf); 
    
    possible_levels := PossibleLevelsForBasechangeLift(f, F);
    "Found", #possible_levels, "possible levels for the lift";

    good_primes := GoodPrimes(f, F, NormUpperbound);
    "Testing the first", #good_primes, "primes";

    possible_lifts := [* *];

    for level in possible_levels do
        H := HilbertCuspForms(F, level, [k : j in [1..n]]);
        Hnew := NewformDecomposition(NewSubspace(H));
        delete H;
        for decomp in Hnew do
            Kh := HeckeEigenvalueField(decomp);
            if Kh is RationalField() or IsIsomorphic(Kh, RationalsAsNumberField()) then
                Kh := RationalsAsNumberField();
            end if;

            if IsSubfield(Kf, Kh) then
                _, emb := IsSubfield(Kf, Kh);
                h := Eigenform(decomp);
                if IsBasechangeLift(h, f, good_primes, emb, automorphisms_Kf) then
                    Append(~possible_lifts, h);
                    if ReturnFirstPossibleLift then
                        return possible_lifts;
                    end if;
                end if;
            end if;
        end for;
    end for;

    if #possible_lifts eq 1 then
        return possible_lifts;
    else
        "Found", #possible_lifts, "possible lifts.";
        "Try increasing the NormUpperbound.";
        return possible_lifts;
    end if;
end function;
---
\end{verbatim}

\end{document}